\documentclass[a4paper]{amsart}
\usepackage[T1]{fontenc}
\usepackage{graphicx}
\usepackage{verbatim}
\usepackage{tikz} 
\usepackage{enumerate}
\usepackage{pdfsync}
\usetikzlibrary{calc} 
\usepackage{algorithm}
\usepackage{algorithmic}

\newcommand{\Det}{\operatorname{Det}}
\newcommand{\Tr}{\operatorname{trace}}
\newcommand{\trace}{\operatorname{trace}}

\newcommand{\I}{\operatorname{I}}
\newcommand{\R}{{\mathbb R}}
\newcommand{\C}{{\mathbb C}}
\newcommand{\E}{{\mathbb E}}

\newcommand{\PP}{{\mathbb P}}

\newcommand{\X}{{\mathcal X}}

\newcommand{\N}{{\mathbb N}}

\newcommand{\F}{{\mathcal F}}

\renewcommand{\d}{{\text{d}}}

\theoremstyle{remark}
\newtheorem*{rem}{Remark}

\theoremstyle{definition}

\theoremstyle{plain}
\newtheorem{thm}{Theorem}[section]

\newtheorem{prop}{Proposition}[section]

\theoremstyle{definition}

\newtheorem{hyp}{Hypothesis}
\newtheorem{defin}{Definition}

\title{Efficient simulation of the Ginibre point process}
\author{L. Decreusefond \and I. Flint \and A. Vergne}
\address{Institut Telecom, Telecom Paristech, CNRS LTCI, Paris, France}
\email{\{laurent.decreusefond,ian.flint,avergne\} @telecom-paristech.fr}
\date{}

\begin{document}

\begin{abstract}
The Ginibre point process is one of the main examples of determinantal point processes on the complex plane. It forms a recurring model in stochastic matrix theory as well as in practical applications. Since its introduction in random matrix theory, the Ginibre point process has also been used to model random phenomena where repulsion is observed. In this paper, we modify the classical Ginibre point process in order to obtain a determinantal point process more suited for simulation. We also compare three different methods of simulation and discuss the most efficient one depending on the application at hand.
\end{abstract}

\maketitle


\section{Introduction}

Determinantal point processes form a class of point processes which exhibit repulsion, and model a wide variety of phenomena. After their introduction by Macchi in \cite{Mb}, they have been studied in depth from a probabilistic point of view in \cite{STa, Sa} wherein we find and overview of their mathematical properties. Other than modeling fermion particles (see the account of the determinantal structure of fermions in \cite{TIa}, and also \cite{Sa} for other examples), they are known to appear in many branches of stochastic matrix theory (see \cite{Sa} or the thorough overview of \cite{AGZa} for example) and in the study of the zeros of Gaussian analytic functions (see \cite{HKPVb}). The Ginibre point process in particular was first introduced in \cite{Ga} and arises in many problems regarding determinantal point processes. To be more specific, the eigenvalues of a hermitian matrix with (renormalized) complex Gaussian entries (which is a subclass of the so-called Gaussian Unitary Ensemble) are known to form a Ginibre point process. Moreover, the Ginibre point process is the natural extension of the Dyson point process to the complex plane. As such, and as explained in \cite{Ga}, it models the positions of charges of a two-dimensional Coulomb gas in a harmonic oscillator potential, at a temperature corresponding to $\beta=2$. It should be noted that the Dyson model is a determinantal point process on $\R$ which is of central importance, as it appears as the bulk-scaling limit of a large class of determinantal point processes, c.f. \cite{Ba}.

Simulation of general determinantal point processes is mostly unexplored, and was in fact initiated in \cite{HKPVa} wherein the authors give a practical algorithm for the simulation of determinantal point processes. Theoretical discussion of the aforementioned algorithm as well as statistical aspects have also been explored in \cite{LMa}. More specifically, the Ginibre point process has spiked interest since its introduction in \cite{Ga}. The simulation procedure which is hinted in \cite{Ga} was fully developed in \cite{LHa}. To the best of our knowledge, the first use of the Ginibre point process as a model traces back to \cite{LDa}. More recently, in \cite{MSa,TLa,DFMVa}, different authors have used the Ginibre point process to model phenomena arising in networking. Indeed, this particular model has many advantages with regards to applications. It is indeed invariant with respect to rotations and translations, which gives us a natural compact subset on which to simulate it: the ball centered at the origin. Moreover, the electrostatic repulsion between particles seems to be fitting for many applications. Our aim in this paper is to study the simulation of the Ginibre point process from a practical point of view, and give different methods which will be more or less suited to the application at hand. The main problem that arises in practice is that although the eigenvalues of matrices in the GUE ensemble form a Ginibre point process, these eigenvalues are \emph{not} compactly supported, although after renormalization, they tend to be compactly supported as $N$ tends to infinity (this is known as the circular law in stochastic matrix theory). Moreover, as will be seen here, truncating to a natural compact and letting $N$ tend to infinity is not the most efficient way to proceed, even though this operation preserves the determinantal property of the point process. Therefore, our methods will rely on the modification of the kernel associated with the Ginibre point process. We study in depth the projection of the kernel onto a compact, its truncation to a finite rank, and in the last part a combination of both operations. Each of these operations on the kernel will have different results on the resulting point process, as well as the simulation techniques involved.

We proceed as follows. We start in Section~\ref{sec:notations} by a general definition of a point process, as well as determinantal point process. We then recall the algorithm from \cite{HKPVa} as well as some more advanced results from \cite{LMa} in Section~\ref{sec:algo}. In Section~\ref{sec:ginibre}, we present more specifically the Ginibre point process, and prove some probabilistic properties. We discuss the truncation, and the projection of the Ginibre kernel and gives the basic ideas that will yield different simulation techniques.

\section{Notations and general results}
\label{sec:notations}

\subsection{Point processes}

Let $E$ be a Polish space, $\mathcal{O}(E)$ the family of all non-empty open subsets of $E$ and $\mathcal{B}$ denotes the corresponding Borel $\sigma$-algebra.  We also consider $\lambda$ a Radon measure on $(E,\mathcal{B})$. Let $\X$ be the space of locally finite subsets in $E$, sometimes called the configuration space:
\begin{equation*}
	\X = \{ \xi \subset E \,: \, | \Lambda \cap \xi | < \infty \, \text{ for any compact set } \Lambda \subset E    \}.
\end{equation*}
In fact, $\X$ consists of all simple positive integer-valued Radon measures (by simple we mean that for all $x \in E$, $\xi({x}) \le 1$). Hence, it is naturally topologized by the vague topology, which is the weakest topology such that for all continuous and compactly supported functions $f$ on $E$, the mapping
\begin{equation*}
	\xi \mapsto \langle f, \xi \rangle := \sum_{y \in \xi} f(y)
\end{equation*}
is continuous. We denote by $\F$ the corresponding $\sigma$-algebra. We call elements of $\X$ configurations and identify a locally finite configuration $\xi$ with the atomic Radon measure $\sum_{y \in \xi} \varepsilon_y$, where we have written $\varepsilon_y$ for the Dirac measure at $y \in E$. 
\\

Next, let $\X_0 = \{ \xi \in \X \, : \, | \xi | < \infty \}$ be the space of all finite configurations on $E$. $\X_0$ is naturally equipped with the trace $\sigma$-algebra $\F_0 = \F |_{\X_0}$.  A random point process is defined as a probability measure $\mu$ on $(\X, \F)$. A random point process $\mu$ is characterized by its Laplace transform $\mathrm{L}_\mu$, which is defined for any measurable nonnegative function $f$ on $E$ as
\begin{equation*}
	{\rm L}_\mu(f) = \int_\X e^{- \sum_{x \in \xi} f(x)} \,\mu(\d \xi)	.
\end{equation*} 
For the precise study of point processes, we also introduce the $\lambda$-sample measure, as well as subsequent tools. Most of our notations are inspired from the ones in \cite{GYa}.
\begin{defin}
	The $\lambda$-sample measure $L$ on $(\X_0, \F_0)$ is defined by the identity
	\begin{equation*}
	\int f(\alpha) \, L(\d \alpha) = \sum_{n \ge 0} \frac{1}{n!} \int_{E^n} f(\{ x_1, \dots, x_n \}) \,\lambda(\d x_1) \dots  \lambda(\d x_n),
	\end{equation*}
	for any measurable nonnegative function $f$ on $\X_0$.
\end{defin}
\noindent
Point processes are often characterized via their correlation function, defined as below.
\begin{defin}[Correlation function]
	A point process $\mu$ is said to have a correlation function $\rho : \X_0 \rightarrow \R$ if $\rho$ is measurable and 
	\begin{equation*}
		\int_\X \sum_{\alpha \subset \xi, \ \alpha \in \X_0} f ( \alpha) \,\mu(\d \xi) = \int_{\X_0} f ( \alpha )\, \rho(\alpha) \, L(\d \alpha) ,
	\end{equation*}
	for all measurable nonnegative functions $f$ on $\X_0$. For $\xi = \{ x_1, \dots, x_n \}$, we will sometimes write $\rho(\xi) = \rho_n(x_1,\dots,x_n)$ and call $\rho_n$ the $n$-th correlation function, where here $\rho_n$ is a symmetrical function on $E^n$. 
\end{defin}
\noindent
It can be noted that correlation functions can also be defined by the following property, both characterizations being equivalent in the case of simple point processes.
\begin{prop}
	A point process $\mu$ is said to have correlation functions $(\rho_n)_{n \in \N}$ if for any $A_1,\dots,A_n$ disjoint bounded Borel subsets of $E$,
	\begin{equation*}
		\E[\prod_{i=1}^n \xi(A_i)] = \int_{A_1\times \dots \times A_n} \rho_n(x_1,\dots,x_n) \,\lambda( \d x_1)\dots \lambda(\d x_n)	.
	\end{equation*}
\end{prop}
\noindent
Recall that $\rho_1$ is the particle density with respect to $\lambda$, and 
\begin{equation*}
	\rho_n(x_1,\dots,x_n) \, \lambda(\d x_1)\dots \lambda(\d x_n)
\end{equation*}
is the probability of finding a particle in the vicinity of each $x_i$, $i=1,\dots,n$.
We also need to define the Janossy density of $\mu$, which is defined as follows:
\begin{defin}
For any compact subset $\Lambda \subseteq E$, the Janossy density $j_\Lambda$ is defined (when it exists) as the density function of $\mu_\Lambda$ with respect to $L_\Lambda$.
\end{defin}
\noindent
In the following, we will write $j_\Lambda^n (x_1,\dots,x_n) = j_\Lambda(\{x_1,\dots,x_n\} ) $ for the $n$-th Janossy density, i.e. the associated symmetric function of $n$ variables, for a configuration of size $n \in \N$. The Janossy density $j_\Lambda(x_1,\dots,x_n)$ is in fact the joint density  (multiplied by a constant)  of the $n$ points given that the point process has exactly $n$ points. Indeed, by definition of the Janossy intensities, the following relation is satisfied, for any measurable $f : \X_0 \rightarrow \R$,
\begin{equation*}
	E[ f(\xi) ] =  \sum_{n\ge 0} \frac{1}{n!} \int_{\Lambda^n} f(\{x_1,\dots,x_n\}) j_\Lambda(\{x_1,\dots,x_n\})\,\lambda(\d x_1)\dots \lambda(\d x_n).
\end{equation*}

\subsection{Determinantal processes}

For details on this part, we refer to \cite{STa,Sa}. For
any compact subset $\Lambda \subset E$, we denote by $L^2(\Lambda,\,
\lambda)$ the set of functions square integrable with respect to the
restriction of the measure $\lambda$ to the set $\Lambda$. This
becomes a Hilbert space when equipped with the usual norm:
\begin{equation*}
  \|f\|^2_{L^2(\Lambda,\lambda)}=\int_{\Lambda}|f(x)|^2 \d\lambda(x).
\end{equation*}
For $\Lambda$ a compact subset of $E$, $P_{\Lambda}$ is the projection
from $L^2(E,\lambda)$ onto $L^2(\Lambda,\lambda)$, i.e., $P_\Lambda f=1_\Lambda.$
The operators we deal with are special cases of the general
set of continuous maps from $L^2(E,\, \lambda)$ into itself.
\begin{defin}
  A map $T$ from $L^2(E,\lambda)$ into itself is said to be an integral
  operator whenever there exists a measurable function, which we still
  denote by $T$, such that
  \begin{equation*}
    Tf(x)=\int_E T(x,\, y) f(y)\d \lambda(y).
  \end{equation*}
  The function $T:E \times E \rightarrow \R$ is called the kernel of $T$.
\end{defin}
\begin{defin}
  Let $T$ be a bounded map from $L^2(E,\, \lambda)$ into itself. The
  map $T$ is said to be trace-class whenever for a complete
  orthonormal basis  $(h_n,\, n\ge 1)$ of $L^2(E,\,
  \lambda)$,
  \begin{equation*}
    \| T \|_1 := \sum_{n\ge 1} (|T | h_n,\, h_n)_{L^2} < \infty,
  \end{equation*}
where $|T| :=  \sqrt{T T^*}$. Then, the trace of $T$ is defined by
  \begin{equation*}
   \Tr (T)= \sum_{n\ge 1} (Th_n,\, h_n)_{L^2}.
  \end{equation*}
\end{defin}
\noindent
It is easily shown that the notion of trace does not depend on the
choice of the complete
  orthonormal basis. Note that if $T$ is trace-class then $T^n$ also is
trace-class for any $n\ge 2$, since we have that $\| T^n \|_1 \le \| T \|^{n-1} \| T \|_1$ (see e.g. \cite{DSa}).
\begin{defin}
  \label{def:fredholm_determinant}
  Let $T$ be a trace-class operator. The Fredholm determinant of
  $(\I+T)$ is defined by:
  \begin{equation*}
    \Det(\I+T)=\exp\left(\sum_{n=1}^{+\infty}\frac{(-1)^{n-1}}{n}\Tr (T^n)\right),
  \end{equation*}
  where $\I$ stands for the identity operator on $L^2(E,\lambda)$.
\end{defin}

\noindent
The Fredholm determinant can also be expanded as a function of the usual determinant, as can be observed in the following proposition, which can be obtained easily by expanding the exponential in the previous definition (see
\cite{STa}):
\begin{prop}
  \label{thm:developpement_det_alpha}
  For a trace-class integral operator $T$, we have:
  \begin{equation*}
    \Det (\I -T)=
    \sum_{n=0}^{+\infty}\frac{1}{n!}\int_{\Lambda^n}{\det\, (T(x_i,\, x_j))_{1\le i,j\le n}}\d\lambda(x_1)\ldots\d\lambda(x_n).
  \end{equation*}
\end{prop}

With the previous definitions in mind, we move onto the precise definition of determinantal point processes. To that effect, we will henceforth use the following set of hypotheses:
\begin{hyp}\label{hyp:condition_T} The map $T$ is an Hilbert-Schmidt
  operator from $L^2(E,\, \lambda)$ into $L^2(E,\, \lambda)$ which
  satisfies the following conditions:
  \begin{enumerate}[i)]
  \item \label{item:1} $T$ is a bounded symmetric integral operator on
    $L^2(E,\, \lambda)$, with kernel $T(.,.)$.
      \item \label{item:2} The spectrum of $T$ is included in $[ 0,\, 1]$.
  \item \label{item:3} The map $T$ is locally of trace-class, i.e.,
    for all compact subsets $\Lambda\subset E$, the restriction
    $T_{\Lambda}:=P_{\Lambda}T P_{\Lambda}$ of $T$ to $L^2(\Lambda,\lambda)$ is
    of trace-class. 
  \end{enumerate}
\end{hyp}
For a compact subset $\Lambda\subset
E$, the map $J[{\Lambda}]$ is defined by:
\begin{equation}
\label{eq:defj}
  J[\Lambda] =\left(\I-
    T_{\Lambda}\right)^{-1} T_{\Lambda},
\end{equation}
so that $T$ and $J[\Lambda]$ are quasi-inverses in the sense that
\begin{equation*}
  \left(\I-T_{\Lambda} \right) \left(\I+
    J[\Lambda]\right) =\I.
\end{equation*}
For any compact $\Lambda$, the operator $J[\Lambda]$ is also a
trace-class operator in $L^2(\Lambda,\, \lambda)$.  In the following
theorem, we define a general determinantal process with three equivalent
characterizations: in terms of their Laplace transforms, Janossy
densities or correlation functions. The theorem is also a theorem of
existence, a problem which is far from being trivial.
\begin{thm}[See \protect{\cite{STa}}]\label{thm:existence}
  Assume Hypothesis~\ref{hyp:condition_T} is satisfied. There exists a unique probability measure
  $\mu_{\,T,\,\lambda}$ on the configuration space $\X$ such
  that, for any nonnegative bounded measurable function $f$ on $E$
  with compact support, we have:
  \begin{equation*}
    L_{\mu_{\, T, \, \lambda}}(f) =\Det \left( \I - T
      [1-e^{-f}]\right),\label{eq:3}
  \end{equation*}
  where $T[1-e^{-f}]$ is the bounded operator on $L^2(E,\lambda)$ with kernel
  :
  \begin{equation*}
    (T[1-e^{-f}])(x,y)=  \sqrt{1-\exp(-f(x))} T(x,y)\, \sqrt{1-\exp(-f(y))}.
  \end{equation*}
  
  This means that for any integer $n$ and any $(x_1,\cdots,\, x_n) \in
  E^n,$ the correlation functions of $\mu_{\, T,\, \lambda}$
  are given by:
  \begin{equation*}
    \rho_{n,\, T}(x_1,\cdots,\,\, x_n)=\det \left( T\left( x_i,\, x_j\right)
    \right)_{1\le i,j\le n},
  \end{equation*}
  and for $n=0$, $\rho_{0,\,T}(\emptyset)=1$.  For any
  compact subset $\Lambda\subset E,$ the operator $J[\Lambda]$ is
  an Hilbert-Schmidt, trace-class operator, whose spectrum is included
  in $[0, +\infty[$.  For any $n \in \N$, any compact $\Lambda\subset
  E$, and any $(x_1,\cdots,\, x_n) \in \Lambda^n$ the $n$-th Janossy
  density is given by:
  \begin{equation}\label{eq:4}
    j_{\Lambda, T}^{n}\left( x_1,\, \cdots,\, x_n\right)
    =\Det\left( \I-
      T_{\Lambda}\right) \det \left(J[\Lambda] (x_i,\, x_j)\right) _{1\le i,j \le n}.
  \end{equation}
  For $n=0$, we have
  \begin{math}
    j_{\Lambda,T}^{0}\left( \emptyset\right)
    =\Det\left( I - T_{\Lambda}\right).
  \end{math}
\end{thm}
We also need a simple condition on the kernels to ensure proper convergence of the associated determinantal measure. This is provided by Proposition 3.10 in~\cite{STa}:
\begin{prop}
\label{prop:convergencefaible}
Let $(T^{(n)})_{n \ge 1}$ be integral operators with nonnegative continuous kernels $T^{(n)}(x,y), \ x, y \in E$. Assume that $T^{(n)}$ satisfy Hypothesis~\ref{hyp:condition_T}, $n \ge 1$, and that $T^{(n)}$ converges to a kernel $T$ uniformly on each compact as $n$ tends to infinity. Then, the kernel $T$ defines an integral operator $T$ satisfying Hypothesis~\ref{hyp:condition_T}. Moreover, the determinantal measure $\mu_{T^{(n)}, \lambda}$ converges weakly to the measure $\mu_{T, \lambda}$ as $n$ tends to infinity.
\end{prop}

\noindent
In the remainder of this section, we shall consider a general determinantal process of kernel $T$ with respect to a reference measure $\lambda$ on $E$. We will assume that $T$ satisfies Hypothesis~\ref{hyp:condition_T}. Consider a compact subset $\Lambda \subset E$. Then, by Mercer's theorem, the projection operator $T_\Lambda$ can be written as
\begin{equation}
\label{eq:kdecomp}
T_\Lambda(x,y) = \sum_{n\ge 0} \lambda_n^\Lambda \varphi_n^\Lambda(x) \overline{ \varphi_n^\Lambda(y) },
\end{equation}
for $x,y \in \C$. Here, $(\varphi_n^\Lambda)_{n \in \N}$ are the eigenvectors of $T_\Lambda$ and $(\lambda_n)_{n \in \N}$ the associated eigenvalues. Note that since $T_\Lambda$ is trace-class, we have 
\begin{equation*}
\sum_{n \ge 0} | \lambda_n^\Lambda| < \infty.
\end{equation*}
In this case, the operator $J[\Lambda]$ defined in \eqref{eq:defj} can be decomposed in the same basis as $T_\Lambda$.
\begin{equation}
\label{eq:jdecomp}
J[\Lambda](x,y) = \sum_{n\ge 0} \frac{\lambda_n^\Lambda}{1-\lambda_n^\Lambda} \varphi_n^\Lambda(x) \overline{ \varphi_n^\Lambda(y) },
\end{equation}
for $x, y \in \Lambda$.
\\

\noindent
Let us conclude this section by mentioning the particular case of the determinantal projection process. We define a projection kernel (onto $\{ \phi_n,\ 0 \le n \le N\} \subset L^2(E,\lambda)  $) to be
\begin{equation}
\label{eq:projectionkernel}
T_p(x,y) =  \sum_{n= 0}^N \varphi_n(x) \overline{ \varphi_n(y) },\quad \forall x,y\in\C
\end{equation}
where $N \in \N$, and $(\varphi_n)_{n \in \N}$ is an orthonormal family of $L^2(E,\lambda)$. We call the associated determinantal process a determinantal projection process (onto $\{ \phi_n,\ 0 \le n \le N\} \subset L^2(E,\lambda)  $). In this case, it is known that the associated determinantal process has $N$ points almost surely, as was first proved in \cite{Sa}. These determinantal processes are particularly interesting since they benefit from a specific simulation technique which will be explained in the next section. 

\section{Simulation of determinantal processes}
\label{sec:algo}

The main results of this section can be found in the seminal work of \cite{HKPVa}, along with the precisions found in \cite{HKPVb} and \cite{LMa}. We recall the algorithm introduced there in order to insist on its advantages and disadvantages compared to directly simulating according to the densities. The idea of the algorithm presented in the previous papers is two-fold. First, it yields a way to simulate the number of points $n\in \N$ of any determinantal process in a given compact $\Lambda \subset E$. Second, it explicits an efficient algorithm for the simulation of the (unordered) density of the point process, conditionally on there being $n$ points, i.e. it yields an efficient algorithm to simulate according to the density $j_\Lambda^n$. Let us now discuss in detail these two steps.
\\

\noindent
The central theorem of this section is proved in \cite[Theorem 7]{HKPVa}. Let us recall it here, as it will be used throughout this paper:
\begin{thm}
\label{thm:thm7}
Let $T$ be a trace-class kernel (we will often take $T_\Lambda$, which is indeed trace-class), which we write
\begin{equation*}
T(x,y) = \sum_{n \ge 1} \lambda_n \varphi_n(x) \overline{ \varphi_n(y)} ,\quad x,y \in E.
\end{equation*}
Then, define $(B_k)_{k \in \N}$ a series (possibly infinite) of independent Bernoulli random variables of mean $\E[ B_k] = \lambda_k$, $k \in \N$. The Bernoulli random variables are defined on a distinct probability space, say $(\Omega,\tilde{F})$. Then, define the (random) kernel
\begin{equation*}
T_B(x,y) =  \sum_{n \ge 1} B_n \varphi_n(x) \overline{ \varphi_n(y)} ,\quad x,y \in E.
\end{equation*}
We define the point process $\eta$ on $(\Xi \times \Omega, \F \otimes \tilde{F})$ as the point process obtained by first drawing the Bernoulli random variables, and then the point process with kernel $T_B$. 
\\

\noindent
Then, we have that in distribution, $\eta$ is a determinantal process with kernel $T$.
\end{thm}
\noindent
For the remainder of this section, we consider a compact subset $\Lambda \subseteq E$, and the associated determinantal process of kernel $T_\Lambda$. We wish to simulate a realization of the aforementioned point process.

\subsection{Number of points}

According to Theorem~\ref{thm:thm7}, the law of the number of points on $E$ has the same law as a sum of Bernoulli random variables. More precisely,
\begin{equation*}
|\xi(E)| \sim \sum_{n \ge 1} B_n,
\end{equation*}
where $B_n \sim \mathrm{Be}(\lambda_n)$, $n \in \N$. Define $T = \sup \{ n \in \N_* \ / \ B_n = 1 \} < \infty$. Since $\sum_{n \ge 1} \lambda_n = \sum_{n \ge 1} \PP(B_n =1) < \infty$, by a direct application of the Borel-Cantelli lemma, we have that $T < \infty$ almost surely. Hence the method is to simulate a realization $m$ of $T$, then conditionally on $T=m$, simulate $B_1,\dots,B_{m-1}$ which are independent of $T$ (note here that $B_m = 1$ almost surely).
\\

\noindent
The simulation of the random variable $T$ can be obtained by the inversion method, as we know its cumulative distribution function explicitly. Indeed, for $n \in \N$, 
\begin{equation*}
\PP(T=n) = \lambda_n \prod_{i = n+1}^\infty (1-\lambda_i),
\end{equation*}
hence 
\begin{equation}
\label{eq:repartitionginibre}
F(t) = \PP(T \le t) = \sum_{n \le t} \lambda_n \prod_{i = n+1}^\infty (1-\lambda_i), \quad \forall t \in \N
\end{equation}
While it is possible to simulate an approximation of the previous distribution function, this requires a numerical approximation of the infinite product, as well as the pseudo-inverse $F^{-1}(u) = \inf\{t \in \N \ / \ F(t) \ge u \}$. We also note that in many practical cases, as is the case with the Ginibre point process, the numerical calculations of the previous functions may well be tedious.
\\

\noindent
Now, assume that we have simulated $B_1,\dots, B_{m-1}, B_m$. If we write $I := \{ 1 \le i \le m \ : \ B_i = 1 \}$, then Theorem~\ref{thm:thm7} assures us that it remains to simulate a determinantal point process with kernel $\sum_{i \in I} \varphi_i(x) \overline{\varphi_i(y) }$, $x, y \in \Lambda$, which has $| I |$ points almost surely. This will be the aim of the next subsection.

\subsection{Simulation of the positions of the points}

Assume we have simulated the number of points $|I| = n \in \N$ according to the previous subsection.  For the clarity of the presentation, we also assume that $B_1=1,\dots, B_n=1$, where $(B_i)_{i \in \N}$ are the Bernoulli random variables defined previously. This assumption is equivalent to a simple reordering of the eigenvectors $(\varphi_i)_{i \in \N}$. Then, conditionally on there being $n$ points, we have reduced the problem to that of simulating the vector $(X_1,\dots,X_n)$ of joint density
\begin{equation*}
p(x_1,\dots,x_n) = \frac{1}{n!} \det\left( \tilde{T}(x_i,x_j)\right)_{1 \le i,j \le n},
\end{equation*}
where $\tilde{T}(x,y) = \sum_{i=1}^n \psi_i(x) \overline{\psi_i(y)}$, for $x,y\in \Lambda$, where here $(\psi_i)_{i \in \N}$ is a reordering of  $(\varphi_i)_{i \in \N}$. The determinantal point process of kernel $\tilde{T}$ has $n$ points almost surely, which means that it remains to simulate the unordered vector $(X_1,\dots,X_n)$ of points of the point process. The idea of the algorithm is to start by simulating $X_n$, then $X_n | X_{n-1}$, up until $X_1 | X_2,\dots,X_{n}$. The key here is that in the determinantal case, the density of these conditional probabilities takes a computable form. Let us start by observing, as is used abundantly in \cite{HKPVa}, that
\begin{equation*}
\det\left( \tilde{T}(x_i,x_j)\right)_{1 \le i,j \le n} = \det\left( \psi_j (x_i) \right)_{1 \le i,j \le n}\det\left( \overline{\psi_i (x_j)} \right)_{1 \le i,j \le n},
\end{equation*}
which allows us to visualize the way the algorithm functions. Indeed, the density of $X_1$ is, for $x_1 \in \Lambda$:
\begin{align*}
p_1(x_1) &= \int \dots \int p(x_1,\dots,x_n) \, \lambda(\mathrm{d}x_2)\dots\lambda(\mathrm{d}x_n)\\
	&= \frac{1}{n!}\sum_{\tau,\sigma \in S_n} \mathrm{sgn}(\tau)\mathrm{sgn}(\sigma) \psi_{\tau(1)}(x_1)\overline{\psi_{\sigma(1)}(x_1)}  \prod_{k=2}^n \int  \psi_{\tau(k)}(x_k)\overline{\psi_{\sigma(k)}(x_k)} \,\lambda(\mathrm{d}x_k)\\
	&= \frac{1}{n!}\sum_{\sigma \in S_n} |\psi_{\sigma(1)}(x_1)|^2\\
	&= \frac{1}{n} \sum_{k=1}^n |{\psi}_{k}(x_1)|^2,
\end{align*}
where $S_n$ is the $n$-th symmetric group and $\mathrm{sgn}(\sigma)$ is the sign of the permutation $\sigma \in S_n$. By the same type of calculations, we can calculate the law of $X_2 | X_1$, whose density with respect to $\lambda$ is given by
\begin{multline*}
p_{2 | X_1} (x_2) = \frac{p_2 (X_1,x_2) }{p_1(X_1)} = \frac{1}{(n-1)!  \sum | {\psi}_i (X_1) |^2  } \\
\shoveright{ \sum_{\sigma \in S_n} \left( |{\psi}_{\sigma(1)}(X_1) |^2 | {\psi}_{\sigma(2)} (x_2) |^2 -  {\psi}_{\sigma(1)}(X_1) \overline{{\psi}_{\sigma(2)}}(X_1){\psi}_{\sigma(2)}(x_2) \overline{{\psi}_{\sigma(1)}}(x_2)  \right) }  \\
	\shoveleft{ = \frac{1}{n-1 } \left( \sum_{i=1}^n | {\psi}_{i} (x_2) |^2 - |  \sum_{i=1}^n \frac{{\psi}_{i}(X_1)}{ \sqrt{\sum | {\psi}_j (X_1) |^2}} \overline{{\psi}_{i} }(x_2)   |^2  \right). }\\
\end{multline*}
The previous formula can be generalized recursively, and has the advantage of giving a natural interpretation of the conditional densities. Indeed, we can write the conditional densities at each step in a way that makes the orthogonalization procedure appear. This is presented in the final algorithm, which was explicited in \cite{LMa} (see also \cite{HKPVa} for the proof). As in \cite{LMa}, we write $\bold v (x) = ( \tilde{\psi}_1(x),\dots, \tilde{\psi}_n(x) )^t$, where $t$ stands for the transpose.
\begin{algorithm}[h]
\caption{Simulation of determinantal projection point process}
 \begin{algorithmic}[H]
 \label{algo:simugen}
   \STATE {\bf sample } $X_n$ from the distribution with density $p_n(x) = \| \bold v(x) \|^2 /n$, $\ x \in \Lambda$ \;
   \STATE $\bold e_1 \leftarrow \bold v(X_n) / \| \bold v(X_n) \|$
   \FOR {$i=n-1 \to 1$}
   \STATE {\bf sample } $X_i$ from the distribution with density \; 
   \begin{equation*}
   p_i(x) = \frac{1}{i} \Big[ \| \bold v(x) \|^2 - \sum_{j = 1}^{n-i} | \bold e_j^* \bold v(x) |^2 \Big]
   \end{equation*}
   \STATE  $\bold w_i \leftarrow \bold v(X_i) - \sum_{j=1}^{n-i} \left(\bold e_j^* \bold v(X_i)\right) \bold e_j,\quad \bold e_{n-i+1} \leftarrow \bold w_i / \| \bold w_i \|$
   \ENDFOR
   \RETURN $(X_1,\dots,X_n)$
 \end{algorithmic}
\end{algorithm}

\noindent
Then, Algorithm~\ref{algo:simugen} yields a sample $(X_1,\dots,X_n)$ which has a determinantal law with kernel $\displaystyle \tilde{T}(x,y) = \sum_{n = 1}^{n} {\psi}(x) \overline{{\psi}(y)}$, $x, y \in \Lambda$. 

\section{Simulation of the Ginibre point process}
\label{sec:ginibre}

\subsection{Definition and properties}

The Ginibre process, denoted by $\mu$ in the remainder of this paper, is defined as the determinantal process on $\C$ with integral kernel
\begin{equation}
\label{eq:ginibre}
K(z_1,z_2) = \frac{1}{\pi} e^{z_1 \bar{z_2}} e^{-\frac{1}{2}( |z_1|^2 + |z_2|^2)},\quad z_1,z_2 \in \C,
\end{equation}
with respect to $\lambda := \d \ell(z)$, the Lebesgue measure on $\C$ (i.e. $\d \ell (z) = \d x\, \d y$, when $z = x + i y$).
It can be naturally decomposed as:
\begin{equation*}
K(z_1,z_2) = \sum_{n \ge0} \phi_n(z_1)\overline{\phi_n(z_2)},\quad z_1, z_2 \in \C,
\end{equation*}
where $\phi_n(z) :=  \frac{1}{\sqrt{\pi n!}} e^{-\frac{1}{2} | z |^2} z^n$, for $n \in \N$ and $z \in \C$. It can be easily verified that $(\phi_n)_{n \in \N}$ is an orthonormal family of $L^2(\C,\mathrm{d}\ell)$. In fact, $(\phi_n)_{n \in \N}$ is a dense subset of $L^2(\C,\d \ell)$. The Ginibre process $\mu$ verifies the following basic properties:
\begin{prop}
The Ginibre process $\mu$, i.e. the determinantal process with kernel $K$ satisfies the following:
\begin{itemize}
\item $\mu$ is ergodic with respect to the translations on the plane.
\item $\mu$ is isotropic.
\item $\mu(\C) = + \infty$ almost surely, i.e. the Ginibre point process has an infinite number of points almost surely.
\end{itemize}
\end{prop}
\begin{proof}
For $a \in \C$, note that $K(z_1 - a,z_2-a) = K(z_1,z_2) e^{-\frac{1}{2} a(\bar{z_2} - \bar{z_1}) + \frac{1}{2}  \bar{a}(z_1-z_2) } $, for $z_1,z_2 \in \C$. Hence, 
\begin{equation*}
\rho(z_1-a,\dots,z_n-a) = \det ( K(z_i-a,z_j-a) )_{1\le i,j \le n} =  \det ( K(z_i,z_j) )_{1\le i,j \le n} ,
\end{equation*}
which means that $\mu$ is invariant with respect to translations. Ergodicity with respect to translations follows from \cite[Theorem 7]{Sa}. \\

\noindent
Moreover, for $\theta \in \R$, we have $K(z_1e^{i \theta},z_2e^{i \theta}) = K(z_1,z_2)$, for $z_1,z_2 \in \C$ (here and in the remainder of the paper, $i := \sqrt{-1}$). Hence, isotropy follows directly by uniqueness of the determinantal measure $\mu$.
\\

\noindent
We have that $\trace K = +\infty$, hence by a classical result (see e.g. Theorem $4$ in \cite{Sa}), the number of points in $\mu$ is almost surely infinite.
\end{proof}

\noindent
Since $\mu$ has an infinite number of points almost surely, it is impossible to simulate it directly. Therefore, in the remainder of this paper, we are interested in modifying the kernel $K$ in order to obtain versions of the Ginibre point process which we can simulate.

\subsection{Truncated Ginibre point process}
The first idea is to consider the truncated Ginibre kernel, defined for $N \in \N_*$ by
\begin{equation}
\label{eq:truncatedginibre}
K^N(z_1,z_2) = \sum_{n =0}^{N-1} \phi_n(z_1)\overline{\phi_n(z_2)},\quad z_1, z_2 \in \C,
\end{equation}
which is in fact a truncation of the sum in \eqref{eq:ginibre}. Additionally, we call $\mu^N$ the associated determinantal point process with intensity measure $\d \ell$. We remark that $\mu^N$ tends to
$\mu$ weakly, when $N$ goes to infinity. As it is a projection kernel of type \eqref{eq:projectionkernel}, we have seen previously that $\mu^N$ has $N$ points almost surely. $\mu^N$ is clearly not translation invariant anymore; however, it remains isotropic for the same reason that $\mu$ is. Physically, $\mu^N$ is the distribution of $N$ polarized electrons in a perpendicular magnetic field, filling the $N$ lowest Landau levels, as is remarked in \cite{SZTa}. As $\mu^{N}$ has $N$ points almost surely, it is entirely characterized by its joint distribution $p$ which is calculated in the following proposition.
\begin{prop}
Let $\mu^N$ be the point process with kernel given by \eqref{eq:truncatedginibre}. Then, $\mu^N$ has $N$ points almost surely and its joint density $p$ is given by
\begin{equation}
\label{eq:loijointe}
p(z_1,\dots,z_N) =  \frac{1}{\pi^N} \prod_{p=0}^{N} \frac{1}{ p!} \,
e^{- \sum_{p=1}^N |z_p|^2}
 \prod_{1\leq p<q \leq N}
| z_p - z_q|^2 ,
\end{equation}
for $z_1,\dots, z_N \in \C$.
\end{prop}
\begin{proof}
\begin{equation*}
p(z_1,\dots,z_N) = \frac{1}{N!} \det \left( K^N\left( x_i,\, x_j\right)
    \right)_{1\le i,j\le N}, \quad z_1,\dots,z_N \in \C,
\end{equation*}
and in this case $p$ can be explicited. Indeed, note that 
\[
p(z_1,\dots,z_N) = \frac{1}{N!}  A^N(z_1,\dots,z_N) A^N(z_1,\dots,z_N)^* ,
\]
 where the matrix $A^N :=(A^N_{ph} )_{1\leq p,h\leq N}$ is given by
\[
A^N_{ph} := \phi_{h-1}(x_p)
\]
and $A^N(x_1,\dots,x_N)^* $ denotes the transpose conjugate of $A^N(x_1,\dots,x_N) $.
Hence,
\[
p(z_1,\dots,z_N)=\frac{1}{N!}  |\mathrm{det}\, A^N(z_1,\dots,z_N) |^2.
\]
We recognize a Vandermonde determinant
\[
\mathrm{det}\, A^N(z_1,\dots,z_N) = \left(\prod_{p=0}^{N-1} \sqrt{\frac{1}{\pi p!}}\right)
e^{-\frac{1}{2} \sum_{p=1}^N |z_p|^2}
 \prod_{1\leq p<q \leq N}
(z_p - z_q),
\]
which leads to the following joint density for the $N$ points:
\begin{equation*}
p(z_1,\dots,z_N) = \frac{1}{\pi^N} \prod_{p=0}^{N} \frac{1}{ p!} \,
e^{- \sum_{p=1}^N |z_p|^2}
 \prod_{1\leq p<q \leq N}
| z_p - z_q|^2 , \quad z_1,\dots,z_N \in \C.
\end{equation*}
\end{proof}

\noindent
It is also known that the radii (in the complex plane) of the points of $\mu^N$ have the same distribution as independent gamma random variables. More precisely, we can find in \cite{Kb} the following result:
\begin{prop}
Let $\{X_1,\dots,X_N\}$ be the $N \in \N_*$ unordered points, distributed according to $\mu^N$. Then, $\{ | X_1 |, \dots, | X_N | \}$ has the same distribution as $\{Y_1,\dots,Y_N\}$, where for $1\le i \le N$, $Y_i ^2 \sim \mathrm{gamma}(i,1)$, and the $Y_i$ are independent.
\end{prop}
\noindent
However, it should be noted that this does not yield a practical simulation technique, as the angles of $X_1,\dots, X_N$  are strongly correlated, and do not follow a known distribution. 
\\

We now move on to the problem of simulating a truncated Ginibre point process with kernel given by \eqref{eq:truncatedginibre}. Since $\mu^N$ has $N$ points almost surely, there is no need to simulate the number of points. One only needs to simulate the positions of the $N$ points.  For this specific case, there is in fact a more natural way of simulating the Ginibre process. Indeed, it was proven in \cite{Ga} that the eigenvalues of an $N\times N$ hermitian matrix with complex gaussian entries are distributed according to $\mu^N$. More precisely, consider a matrix $N := (N_{n m})_{1 \le n,m \le N}$, such that for $1 \le n,m \le N$, 
\begin{equation*}
N_{n m} = \frac{1}{\sqrt{2}} \left(N_{n m}^1 + i N_{n m}^2\right),
\end{equation*}
where $N_{n m}^1,N_{n m}^2 \sim \mathcal{N}(0,1)$, $1 \le n,m \le N$ are independent centered gaussian random variables. Then, the eigenvalues of $N$ are distributed according to $\mu^N$. This is by far the most efficient way of simulating the truncated Ginibre process.
\\

\noindent
We also remark that we could have applied the simulation technique of Section~\ref{sec:algo} in order to simulate the truncated Ginibre point process. However, the simulation procedure is much slower than calculating the eigenvalues of an $N \times N$ matrix. We still show the results of the algorithm of a realization of the resulting point process in the following. This allows proper visualization of the associated densities. We chose $a=3$ and $N = 8$ in this example. We plot the densities $p_i$ as color gradients before the simulation of the $(N-i)$-th point. The steps plotted in the following figure correspond to $i=7, i=4$, and $i=1$ respectively (Algorithm~\ref{algo:simu} is used and is run from $i = N$ to $i=1$). We also mark by red points the previously simulated points. Therefore, the point process obtained at the end of the algorithm consists of the red points in the third figure. 
\begin{figure}[!h]
\minipage{0.3\textwidth}
  \includegraphics[width=\linewidth]{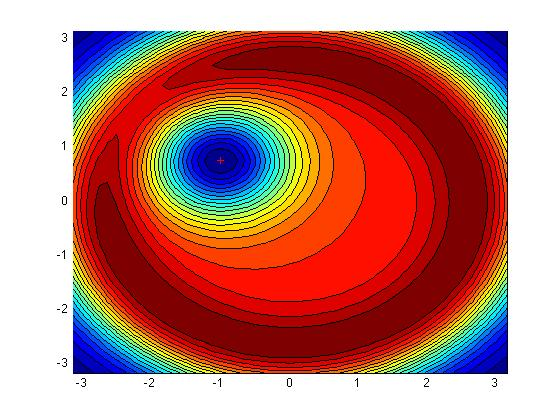}
\endminipage\hfill
\minipage{0.3\textwidth}
  \includegraphics[width=\linewidth]{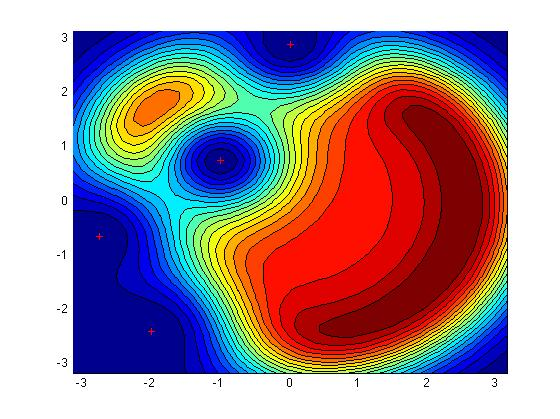}
\endminipage\hfill
\minipage{0.3\textwidth}%
  \includegraphics[width=\linewidth]{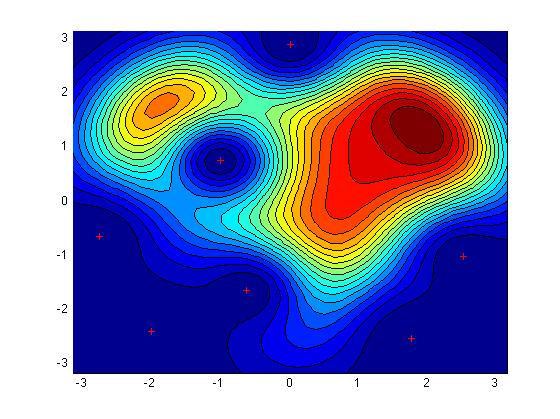}
\endminipage
\end{figure}
\\
\\
\\

\noindent
However, one runs into a practical problem when simulating the truncated Ginibre process: the support of its law is the whole of $\C^N$. Recall that the joint law of $\mu^N$ is known to be given by \eqref{eq:loijointe} which has support on $\C^N$. Moreover, projecting onto a compact subset randomizes the number of points in the point process. Therefore, this first method is only useful in applications where the point process need not be in a fixed compact subset of $E$. 

\subsection{Ginibre point process on a compact subset}

We now consider more specifically the projection of the Ginibre process onto $\mathcal{B}_R$, and thus we consider the projection kernel $K_{R}:=P_{\mathcal{B}_R}K P_{\mathcal{B}_R}$ of the integral operator $K$ onto $\mathrm{L}^2(\mathcal{B}_R, \d \ell)$, where $\mathcal{B}_R := \overline{\mathcal{B}(0,R)}$ is the closed ball of $\C$ of radius $R\ge 0$ with center $0$. In this specific case, the kernel of the operator $K_R$ takes the form:
\begin{equation}
\label{eq;ginibrecompact}
K_R(z_1,z_2) =  \sum_{n \ge0} \lambda_n^R\phi_n^R(z_1)\overline{\phi_n^R(z_2)},
\end{equation}
where $\phi_n^R(z) := Z_{R,n}^{-1} \phi_n(z) 1_{z \in \mathcal{B}_R}$, $n \in \N$, $z \in \C$ and $Z_{R,n}^{-1}\in\R$ is a constant depending only on $n$. This result does not hold in general, but is due to the fact that $(\phi_n^R(\cdot))_{n \ge 0}$ is still an orthonormal family of $L^2(\mathcal{B}_R,\mathrm{d}z)$. Indeed, for $m,n \in \N$,
\begin{align*}
\int_{\mathcal{B}_R} \phi_n^R(z) \overline{\phi_m^R(z)} \, \d \ell(z) &= Z_{R,n}^{-2}  \left( \frac{1}{\sqrt{n! m!}}  \int_0^R r^{n+m+1} e^{-r^2} \,\mathrm{d}r  \right)\left(\frac{1}{\pi} \int_{-\pi}^\pi e^{i(n-m)\theta}\,\mathrm{d}\theta \right)\\
&=  Z_{R,n}^{-2} 1_{n = m} \left( \frac{2}{n!}  \int_0^R r^{2n+1} e^{-r^2} \,\mathrm{d}r  \right)\\
&= Z_{R,n}^{-2} 1_{n = m} \frac{\gamma(n+1,R^2)}{n!} ,
\end{align*}
where $\gamma$ is the lower incomplete Gamma function defined as
\begin{equation*}
\gamma(z,a) := \int_0^a e^{-t} t^{z-1}\,\mathrm{d}t,
\end{equation*}
for $z \in \C$ and $a \ge0$. Hence, in the following, we shall take $Z_{R,n} := \sqrt{ \frac{\gamma(n+1,R^2)}{n!} }$. Therefore, the associated eigenvalues are
\begin{equation*}
\lambda_n^R := \int_{\mathcal{B}_R} | \phi_n(z)|^2 \,\d \ell(z) = Z_{R,n}^2 =  \frac{\gamma(n+1,R^2)}{n!}.
\end{equation*}
As is expected, $0\le\lambda_n^R \le 1$ for any $n\in\N, R \ge 0$, and $\lambda_n^R \xrightarrow[R \rightarrow \infty]{} 1$ for any $n \in \N$.
\\

\noindent
Now that we have specified the eigenvectors and associated eigenvalues, the simulation of the Ginibre process on a compact is that of the determinantal point process with kernel given by \eqref{eq;ginibrecompact}. Therefore, Algorithm~\ref{algo:simugen} fully applies. The time-consuming step of the algorithm will be the simulation of the Bernoulli random variables. Recall that the cumulative distribution function of $T = \sup \{ n \in \N_* \ / \ B_n = 1 \}$ is given by \eqref{eq:repartitionginibre} which in our case is equal to
\begin{equation*}
F(m) =  \sum_{n \le m} \frac{\gamma(n+1,R^2)}{n!} \prod_{i = n+1}^\infty \frac{\Gamma(i+1,R^2)}{i!},
\end{equation*}
for $m \in \N_*$. 
\\

\noindent
We remark that we can not simulate the Ginibre point process restricted to a compact  in the same way as in the previous subsection. Indeed, taking a $N \times N$ matrix with complex gaussian entries, and conditioning on the points being in $\mathcal{B}_{R}$ yields a determinantal point process with kernel,
\begin{equation*}
K_R(z_1,z_2) =  \sum_{n =0}^{N-1} \lambda_n^R\phi_n^R(z_1)\overline{\phi_n^R(z_2)},
\end{equation*}
which is not our target point process, as the sum is truncated at $N$. Therefore, the method developed in the previous subsection does not apply here. Hence, the algorithm is twofold, and the first step goes as follows:
\begin{algorithm}[h]
\caption{Simulation of the Ginibre process on a compact subset (Step $1$)}
 \begin{algorithmic}[H]
 \label{algo:simucompact1}
   \STATE {\bf evaluate numerically } $R \leftarrow \prod_{i\ge 1} \frac{\Gamma(i+1,R^2)}{i!}$, for example by calculating $\displaystyle e^{\sum_{i= 1}^N \ln(\frac{\Gamma(i+1,R^2)}{i!}) }$, where $N$ is chosen such that $\ln(\frac{\Gamma(N+1,R^2)}{N!}) < \epsilon$, $\epsilon > 0$ given by the user. \;
   \STATE {\bf sample } $U \leftarrow \mathcal{U}([0,1])$ according to a uniform distribution on $[0,1]$. \;
    \STATE $m \leftarrow 0$
   \WHILE {$U < R$}
   \STATE $ m \leftarrow m + 1$ \;
   \STATE $R \leftarrow \frac{m! \gamma(m+1,R^2)}{\gamma(m,R^2) \Gamma(m+1,R^2)}R$  \; 
   \ENDWHILE
\FOR { $i=0 \rightarrow m-1$ }
 \STATE $B_i \leftarrow \mathrm{Be}(\frac{\gamma(i+1,R^2)}{i!})$, where here $\mathrm{Be}(\lambda)$ is an independent drawing of a Bernoulli random variable of parameter $\lambda \in [ 0,\, 1]$
\ENDFOR
\IF { $m > 0$ }
  \RETURN $\{B_0,\dots,B_{m-1}, 1\}$ \;
\ENDIF
\IF { $m = 0$ }
  \RETURN $\{1\}$ \;
\ENDIF
 \end{algorithmic}
\end{algorithm}
\\

\begin{rem}
The series $\prod_{i\ge n} \frac{\Gamma(i+1,R^2)}{i!}$, for $n \in \N_*$ converges since it is equal to $\prod_{i\ge n} (1- \lambda_i^R)$ which is convergent. Indeed, $\sum_{i \ge 0} \lambda_i^R < \infty$ since the considered operator is locally trace-class.
\end{rem}

\noindent
We write  $\{B_0,\dots,B_{m-1}, 1\}$ for the value returned by the previous algorithm, with the convention that $\{B_0,\dots,B_{m-1},1\} = \{1\}$ if $m = 0$. Then by Theorem~\ref{thm:thm7}, the law of the Ginibre point process on a compact is the same as that of the determinantal point process of kernel 
\begin{equation*}
K(z_1,z_2) = \sum_{k = 0}^m B_k \phi_n(z_1)\overline{\phi_n(z_2)},\quad z_1, z_2 \in \C.
\end{equation*}

Now, we move onto the second part of the algorithm, which is this time straightforward as it suffices to follow Section~\ref{sec:algo} closely. 
\begin{algorithm}[h]
\caption{Simulation of the Ginibre process on a compact subset (Step $2$)}
 \begin{algorithmic}[H]
 \label{algo:simucompact2}
   \STATE {\bf define } $\phi_k(z) := \frac{1}{\pi  \gamma(k+1,R^2)}e^{- \frac{1}{2 } | z |^2} {z}^k $, for $z \in \mathcal{B}_R$ and $0\le k \le m$. \;
   \STATE {\bf define } $\bold v(z) := (\phi_{i_0}(z),\dots,\phi_{i_{k}}(z),\phi_{m}(z)  ) $, for $z \in \mathcal{B}_a$, and where $\{ i_0,\dots,i_{k}\} = \{ 0 \le i \le m-1 \ : \ B_i = 1 \}$ \;
   \STATE $N \leftarrow k+2$ \;
   \STATE {\bf sample } $X_N$ from the distribution with density $p_N(x) = \| \bold v(x) \|^2 /N$, $\ x \in \Lambda$ \;
   \STATE  $\bold e_1 \leftarrow \bold v(X_N) / \| \bold v(X_N) \|$
   \FOR {$i=N-1 \to 1$}
   \STATE {\bf sample } $X_i$ from the distribution with density \; 
   \begin{equation*}
   p_i(x) = \frac{1}{i} \Big[ \| \bold v(x) \|^2 - \sum_{j = 1}^{N-i} | \bold e_j^* \bold v(x) |^2 \Big]
   \end{equation*}
   \STATE $\bold w_i \leftarrow \bold v(X_i) - \sum_{j=1}^{N-i} \left(\bold e_j^* \bold v(X_i)\right) \bold e_j,\quad \bold e_{N-i+1} \leftarrow \bold w_i / \| \bold w_i \|$
   \ENDFOR
   \RETURN $(X_1,\dots,X_N)$
 \end{algorithmic}
\end{algorithm}
\\

\noindent
We end this subsection by mentioning the difficulties arising in the simulation under the density $p_i$, $1 \le i \le N-1$. As is remarked in \cite{LMa}, in the general case, we have no choice but to simulate by rejection sampling and the Ginibre point process is no different (except in the case $i=N-1$ where $p_i$ is the density of a gaussian random variable). Therefore in practice, we draw a uniform random variable $u$ on $\mathcal{B}_a$ and choose $p_i(u)/\sup_{y \in \mathcal{B}_a} p_i(y)$. Note that the authors in \cite{LMa} give a closed form bound on $p_i$ which is given by
\begin{equation}
\label{eq:boundpi}
p_i(x) \le \frac{1}{i} \min_{i+1 \le k \le N} \left( K^N(x,x) - \frac{| K^N(x,X_k)|^2}{K^N(X_k,X_k)} \right),
\end{equation}
where $X_{i+1},\dots,X_N$ is the result of the simulation procedure up to step $i$. In practice however, the error made in the previous inequality is not worth the gain made by not evaluating $\sup_{y \in \mathcal{B}_a} p_i(y)$. Therefore, in our simulations, we have chosen not to use \eqref{eq:boundpi}.

\subsection{Truncated Ginibre process on a compact subset}
In this subsection, we begin by studying the truncated Ginibre point process on a compact subset, and specifically discuss the optimal choice of the compact subset onto which we project. We begin by studying the general projection of the truncated Ginibre process onto a centered ball of radius $R \ge 0$ which is again a determinantal point process whose law can be explicited. To that end, we wish to study $K_{R}^N :=P_{\mathcal{B}_R}K^N P_{\mathcal{B}_R}$ of the integral operator $K$ onto $\mathrm{L}^2(\mathcal{B}_R, \d \ell)$. The associated kernel is given by 
\begin{equation}
\label{eq:truncatedproj}
K_R^N(z_1,z_2) =  \sum_{n = 0}^{N-1}  \lambda_n^R\phi_n^R(z_1)\overline{\phi_n^R(z_2)},
\end{equation}
for $z_1, z_2 \in \mathcal{B}_R$. The question of the Janossy densities of the associated determinantal process is not as trivial as the non-projected one. Indeed, $\mu_R^N$ does not have $N$ points almost surely. However, it is known that it has less than $N$ points almost surely (see e.g. \cite{Sa}). Therefore, it suffices to calculate the Janossy densities $j^0_R,\dots,j_R^N$ to characterize the law of $\mu_R^N$. These are given by the following proposition:
\begin{prop}
The point process $\mu_R^N$ with kernel given by \eqref{eq:truncatedproj} has less than $N$ points almost surely, and its Janossy densities are given by 
\begin{multline*}
j_R^k(z_1,\dots,z_k)=\frac{1}{\pi^k}  \prod_{p=0}^{k-1} \frac{1}{ p! } e^{- \sum_{p=1}^k |z_p|^2} \prod_{1 \le i < j \le k} | z_i - z_j |^2 
\\
\sum_{\{i_1,\dots,i_k\} \subset \{1,\dots, N\}} 
 | s_{\lambda(i_1,\dots,i_k)}(z_1,\dots,z_k)|^2,
\end{multline*}
for $0 \le k \le N$ and $z_1,\dots,z_k \in \mathcal{B}_R$.

\end{prop}
\begin{proof}
By formula \eqref{eq:jdecomp}, the operator $J^N[\mathcal{B}_R]$ associated to $\mu^N$ can be decomposed as:
\begin{equation*}
J^N[\mathcal{B}_R] (z_1,z_2) =  \sum_{n = 0}^{N-1} \frac{\gamma(n+1,R^2)}{\Gamma(n+1,R^2)} \,\phi_n^R(z_1)\overline{\phi_n^R(z_2)}, \quad z_1, z_2 \in \mathcal{B}_R,
\end{equation*}
where $\Gamma$ is the upper incomplete Gamma function defined as
\begin{equation*}
\Gamma(z,a) := \int_a^\infty e^{-t} t^{z-1}\,\mathrm{d}t,
\end{equation*}
for $z \in \C$ and $a \ge0$, which by definition verifies $\gamma(\cdot, a) + \Gamma(\cdot,a) = \Gamma(\cdot)$ for all $a \ge 0$ ($\Gamma(\cdot)$ is the usual Gamma function). Here, we note that $j^N_R$ can be calculated as previously as the associated determinant is again a Vandermonde determinant. More precisely, we obtain
\begin{equation*}
\det \left( J^N [ \mathcal{B}_R ] (z_i, z_j) \right)_{1 \le i, j \le N } = \frac{1}{\pi^N} \prod_{p=0}^{N-1} \frac{1}{ \Gamma(p+1, R^2)} \,
e^{- \sum_{p=1}^N |z_p|^2}
 \prod_{1\leq p<q \leq N}
| z_p - z_q|^2 , 
\end{equation*}
for $z_1,\dots,z_N \in \mathcal{B}_R$. Moreover, the hole probability, i.e. the probability of having no points in $\mathcal{B}_R$, is equal to 
\begin{equation}
\label{holeproba}
\mathrm{Det} \left( \I - K_R^N \right) = \prod_{ n=0}^{N-1} (1-\lambda_n^R) = \prod_{n =0}^{N-1} \frac{\Gamma(n+1, R^2)}{n!}.
\end{equation}
Hence, we obtain the following expression for the $N$-th Janossy density:
\begin{equation*}
j_R^N (z_1,\dots,z_N) =  \frac{1}{\pi^N} 
 \prod_{p=0}^{N-1} \frac{1}{p!} \,
e^{- \sum_{p=1}^N |z_p|^2}
 \prod_{1\leq p<q \leq N}
| z_p - z_q|^2  ,
\end{equation*}
for $z_1,\dots,z_N \in \mathcal{B}_R$. Now, if we take $k < N$, we have again
\[
J[D](z_1,\dots,z_k)=A^N(z_1,\dots,z_k)A^N(z_1,\dots,z_k)^*,
\]
where this time, $A^N(z_1,\dots,z_k) $ is a rectangular $k\times N$ matrix. Hence, by application of the Cauchy-Binet formula:
\[
\mathrm{det}\,J[D](z_1,\dots,z_k)=\sum_{\{i_1,\dots,i_k\} \subset \{1,\dots, N\}} | \mathrm{det}\, A^{i_1,\dots,i_k}(z_1,\dots,z_k)|^2,
\]
where we have for $1\le p,h \le k$,
\[
A^{i_1,\dots,i_k}_{ph}(z_1,\dots,z_k) := \sqrt{\frac{\gamma(n+1,R^2)}{\Gamma(n+1,R^2)}} \phi_{i_h}^{R}(z_p),
\]
which is a square matrix. We now consider fixed ${\{i_1,\dots,i_k\} \subset \{1,\dots, N\}}$ and wish to evaluate $| \mathrm{det}\, A^{i_1,\dots,i_k}(z_1,\dots,z_k)|^2$. In fact, we observe that
\[
| \mathrm{det}\, A^{i_1,\dots,i_k}(z_1,\dots,z_k)|^2 = \prod_{p=0}^{k-1} \frac{1}{\pi \Gamma(p+1, R^2) }
e^{- \sum_{p=1}^k |z_p|^2}
\left|
V_{i_1,\dots,i_k}(z_1,\dots,z_k)
\right|^2,
\]
where 
\[
V_{i_1,\dots,i_k}(z_1,\dots,z_k) := \mathrm{det}\left( \left(z_h^{i_p}\right)_{1\le p,h \le k} \right)
\]
is known in the literature as the generalized Vandermonde determinant. Here, $V_{1,\dots,k}(z_1,\dots,z_k)$ is the classical Vandermonde determinant, and in the general case,  a certain number of rows from the matrix have been deleted. The generalized Vandermonde determinant is known to factorize into the classical Vandermonde determinant and what is defined to be a Schur polynomial. To be more precise,
\[
V_{i_1,\dots,i_k}(z_1,\dots,z_k)  = V_{1,\dots,k}(z_1,\dots,z_k)s_{\lambda(i_1,\dots,i_k)}(z_1,\dots,z_k), 
\]
where ${\lambda(i_1,\dots,i_k)}:= (i_k-k+1,\dots,i_2-1,i_1)$, and $s_{\lambda}$ is the Schur polynomial, which is known to be symmetric, and is a sum of monomials, see e.g. \cite{Ha}. To summarize, we have 
\begin{multline*}
\mathrm{det}\,J[D](z_1,\dots,z_k)=\left( \prod_{p=0}^{k-1} \frac{1}{\pi \Gamma(p+1,R^2) }\right) e^{- \sum_{p=1}^k |z_p|^2} \prod_{1 \le i < j \le k} | z_i - z_j |^2 \\
 \sum_{\{i_1,\dots,i_k\} \subset \{1,\dots, N\}} 
 | s_{\lambda(i_1,\dots,i_k)}(z_1,\dots,z_k)|^2.
\end{multline*}
Then, by \eqref{holeproba}, we find
\begin{multline}
\label{eq:detjginibre}
j_R^k(z_1,\dots,z_k)=\frac{1}{\pi^k}  \prod_{p=0}^{k-1} \frac{1}{ p! } e^{- \sum_{p=1}^k |z_p|^2} \prod_{1 \le i < j \le k} | z_i - z_j |^2 
\\
\sum_{\{i_1,\dots,i_k\} \subset \{1,\dots, N\}} 
 | s_{\lambda(i_1,\dots,i_k)}(z_1,\dots,z_k)|^2,
\end{multline}
for $z_1,\dots, z_k \in \mathcal{B}_R$. 
\end{proof}

Next, we wish to determine the optimal $R \ge 0$ onto which we project the truncated Ginibre process. In regards to this question, we recall that the particle density $\rho_1$ of the general Ginibre process is constant, and 
\begin{equation*}
\rho_1(z) = K(z,z) = \frac{1}{\pi},
\end{equation*}
for $z \in \C$. However, the particle density of the truncated Ginibre process is not constant. If we denote by $\rho_n^N$ the $n$-th correlation function of $\mu^N$, then we have
\begin{equation*}
\rho_1^N (z) = \frac{1}{\pi} e^{-\frac{1}{2}|z|^2} \sum_{k = 0}^{N-1} \frac{ |z|^{2 k}}{k!}   ,
\end{equation*}
for $z \in \C$. As can be checked easily, we have $\int_\C \rho_1^N(z)\,\d z = N$ as well as 
\begin{equation}
\label{eq:rhoinfpi}
\rho_1^N(z) \le \frac{1}{\pi},\quad  \forall z \in \C,
\end{equation}
and in fact it is known that $\rho_1^N(\sqrt{N} z) \xrightarrow[N \rightarrow \infty]{} \frac{1}{\pi} 1_{| z | \le 1}$, which is known as the circular law in stochastic matrix theory. It therefore appears that it is optimal to project onto $\mathcal{B}_{\sqrt{N}}$. We wish to get more precise results on the error we are making by truncating the point process to $\mathcal{B}_R$. To that end, we recall the following bounds on $\rho_1^N$ which were obtained in \cite{Ga}. We recall their proof for convenience, as our bounds differ slightly from the ones obtained there.
\begin{prop}
\label{prop:densiteginibre}
For $|z|^2 < N+1$, we have
\begin{equation*}
\frac{1}{\pi} - \rho_1^N (z) \le \frac{1}{\pi} e^{-|z|^2} \frac{|z|^{2N}}{N!} \frac{N+1}{N+1 - |z|^2}.
\end{equation*}
For $|z|^2 \ge N$, we have
\begin{equation*}
\rho_1^N (z) \le \frac{1}{\pi} e^{-|z|^2} \frac{|z|^{2N}}{N!} \frac{N}{ |z|^2 - N}.
\end{equation*}
\end{prop}
\begin{proof}
By using $\frac{(k+N)!}{N!} \ge (N+1)^k$, for $k, N \in \N$, we obtain for  $|z|^2 < N+1$,
\begin{align*}
\rho_1^N (z) = \frac{1}{\pi} -  \frac{1}{\pi} e^{-|z|^2} \sum_{k=N}^\infty \frac{|z|^{2k}}{k!} &= \frac{1}{\pi} -  \frac{1}{\pi} e^{-|z|^2} \frac{|z|^{2N}}{N!} \sum_{k=0}^\infty \frac{|z|^{2k} N!}{(k+N)!}\\
	&\le \frac{1}{\pi} -  \frac{1}{\pi} e^{-|z|^2} \frac{|z|^{2N}}{N!} \frac{1}{1-\frac{|z|^2}{N+1}}.
\end{align*}
The proof of the second inequality is along the same lines, except that we use $ \frac{N!}{(n-k)!} \le N^k$, for $k, N \in \N$.
\end{proof}
As was noticed in \cite{Ga}, if we set $|z| = \sqrt{N} + u$, for $-1 \le u \le 1$, both of the right hand sides of the inequalities in Proposition~\ref{prop:densiteginibre} tend to 
\begin{equation*}
\frac{1}{2 \sqrt{2} u \pi^{3/2} } e^{-2 u^2},
\end{equation*}
as $N$ tends to infinity. This is obtained by standard calculations involving in particular the Stirling formula. That is to say, for $|z| \le \sqrt{N}$, and $z = \sqrt{N} - u$, 
\begin{equation}
\label{eq:borneinf}
\rho_1^N (\sqrt{N} - u) \ge \frac{1}{\pi} - \frac{1}{2 \sqrt{2} u \pi^{3/2} } e^{-2 u^2},
\end{equation}
as well as for $|z| \ge \sqrt{N}$, and $z = \sqrt{N} + u$
\begin{equation}
\label{eq:bornesup}
\rho_1^N (\sqrt{N} + u) \le \frac{1}{2 \sqrt{2} u \pi^{3/2} } e^{-2 u^2},
\end{equation}
as $N$ tends to infinity. These bounds exhibit the sharp fall of the particle density around $|z| = \sqrt{N}$.
\begin{figure}[h!]
  \caption{$\rho_1^N (|z|)$ for $N = 600$ and $|z|$ around $\sqrt{N}$ (blue). Upper and lower bounds obtained in (\ref{eq:borneinf}) and (\ref{eq:bornesup}) (green).}
  \centering
    \includegraphics[width=\textwidth]{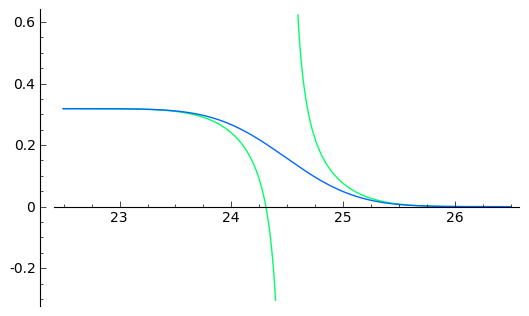}
\end{figure}
\\

\noindent
The previous results yield in particular the next proposition (see \cite{Ga}). We give here its proof as it was omitted in \cite{Ga}.
\begin{prop}
\label{prop:eqdelta}
Let us write $\displaystyle \delta (N) := \int_{|z| > \sqrt{N}} \rho_1^N(z) \,\d \ell(z)$. Then, we have 
\begin{equation}
\label{eq:erreurapprox}
\delta (N) \sim \sqrt{\frac{N}{ 2 \pi}},
\end{equation}
as $N \rightarrow \infty$.
\end{prop}
\begin{proof}
For any $0 < a < \sqrt{N}$, we define $\delta_a (N) := \int_{|z| > \sqrt{N} + a} \rho_1^N(z) \,\d \ell(z)$. As the bound obtained in \eqref{eq:bornesup} is not integrable at $u =0$, we bound the two different parts as follows:
\begin{equation*}
\delta(N) = \delta_a (N) + \int_{\sqrt{N} \le |z| \le \sqrt{N} + a} \rho_1^N(z) \,\d \ell(z) \le \delta_a (N) + 2 a \sqrt{N} + a^2,
\end{equation*}
where we have used \eqref{eq:rhoinfpi}. Now, applying \eqref{eq:bornesup} yields the following:
\begin{equation*}
\delta_a(N) \le \frac{1}{4} -\frac{1}{2} \sqrt{\frac{N}{2 \pi}} E_i(-2a^2),
\end{equation*}
where here $E_i(\cdot)$ stands for the exponential integral defined as
\begin{equation*}
E_i(x) := - \int_{-x}^\infty \frac{1}{t} e^{-t} \, \d t, \quad x \in \R_+^*. 
\end{equation*}
To sum up the calculations up to this point, we have for $N$ sufficiently large,
\begin{equation*}
\frac{1}{\sqrt{N}} \delta(N) \le \frac{a^2 + \frac{1}{4}}{\sqrt{N}} + 2 a -\frac{1}{2} \sqrt{\frac{1}{2 \pi}} E_i(-2a^2).
\end{equation*}
On the other hand, we also have
\begin{equation*}
\delta(N) = N - \int_{|z| \le \sqrt{N}} \rho_1^N(z) \,\d \ell(z) \ge  N - \int_{|z| \le \sqrt{N}-a} \rho_1^N(z) \,\d \ell(z) +a^2 - 2 a \sqrt{N},
\end{equation*}
thanks to \eqref{eq:rhoinfpi}. This time, we use \eqref{eq:borneinf} and obtain:
\begin{equation*}
\delta(N) \ge - \frac{1}{4} + \sqrt{N} \int_a^{\sqrt{N}} \frac{1}{\sqrt{2 \pi} u } e^{- 2 u^2} \,\d u,
\end{equation*}
which means that for sufficiently large $N$, 
\begin{equation*}
\delta(N) \ge  - \frac{1}{4} - \frac{1}{2} \sqrt{\frac{N}{2 \pi}} E_i(-2a^2).
\end{equation*}
Therefore, we have the following bounds, as $N \rightarrow \infty$.
\begin{equation*}
-\frac{1}{2} E_i(-2a^2) \le \sqrt{\frac{2 \pi}{N}} \delta(N) \le 2 \sqrt{2 \pi} a -\frac{1}{2} E_i(-2a^2).
\end{equation*}
However, since $E_i(-2a^2) \sim 2 \log a$ as $a$ tends to $0$, we have that $\sqrt{\frac{2 \pi}{N}} \delta(N) \rightarrow 1$ as $N \rightarrow \infty$ by taking a small enough $a$.
\end{proof}

\noindent
Proposition~\ref{prop:eqdelta} means that as $N \rightarrow \infty$, the average number of points falling outside of the $\mathcal{B}_{\sqrt{N}}$ is of the order of $\frac{1}{\sqrt{N}}$, as $N$ tends to infinity. Therefore, from now on, we will consider the truncated Ginibre process of rank $N$ projected onto $\mathcal{B}_{\sqrt{N}}$. Assume that we need to simulate $\mu^N$ on a compact subset. Then, we no longer control the number of points, i.e. there is again a random number of points in the compact subset, as seen in Figure~\ref{fig:truncatedginibre}.

\begin{figure}[h!]
  \caption{A realization of $\mu^N$ for $N = 200$ (blue circles) renormalized to fit in the circle of radius $1$ (in red)}
    \includegraphics[width=1.7\textwidth]{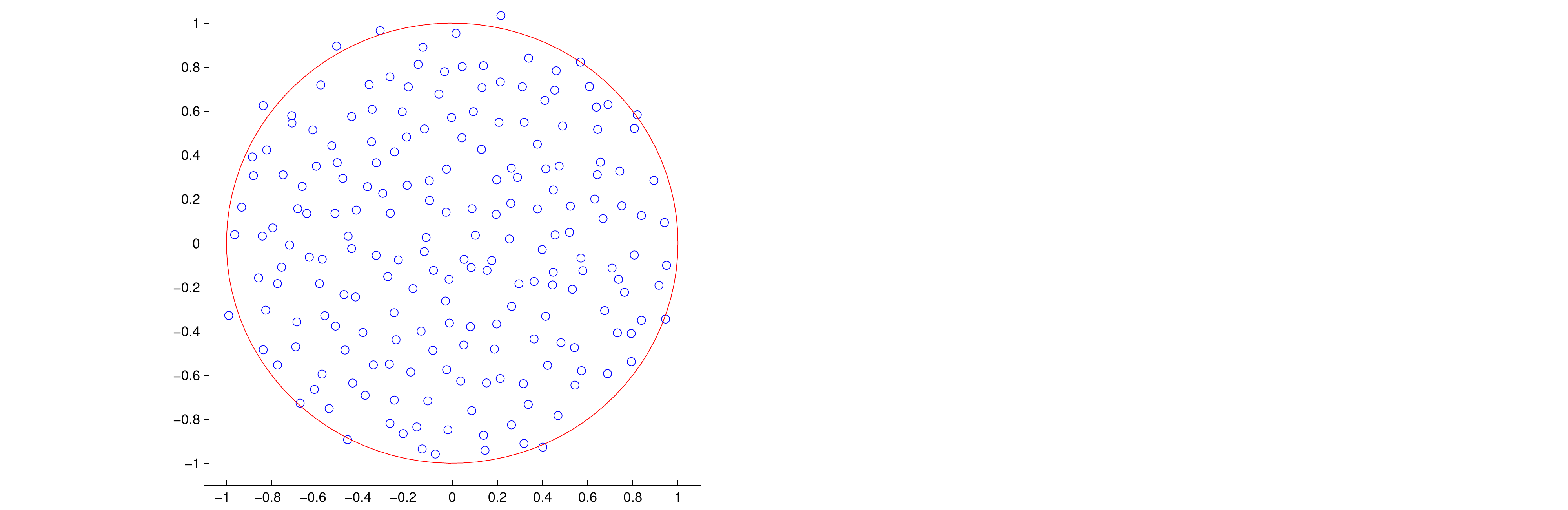}
    \label{fig:truncatedginibre}
      \centering
\end{figure}

Therefore, our additional idea is to condition the number of points on being equal to $N$. As we have calculated previously in Proposition~\ref{prop:eqdelta}, there is a number of points falling outside of the ball of radius $\mathcal{B}_{\sqrt{N}}$ which grows as $\frac{1}{\sqrt{N}}$ as $N$ goes to infinity. Since the projection onto $\mathcal{B}_{\sqrt{N}}$ of the truncated Ginibre process takes the determinantal form \eqref{eq:truncatedproj}, one can easily calculate the probability of all the points falling in $\mathcal{B}_{\sqrt{N}}$. Indeed, we have that
\begin{equation}
\label{eq:truncatedinball}
\PP_{\mu^N} ( \xi_{\mathcal{B}_{\sqrt{N}}^c} = \emptyset ) = \prod_{n = 0}^{N-1} \lambda_n^N =  \prod_{n = 0}^{N-1} \frac{\gamma(n+1,N)}{n!}.
\end{equation}
It can be shown that this probability tends to $0$ as $N$ tends to infinity. That is, if we are required to simulate the Ginibre process on a compact conditionally on it having $N$ points, the conditioning requires more and more computation time as $N$ tends to infinity. 
\\

\noindent
However, we are not forced to simulate the conditioning on there being $N$ points. Instead, we introduce a new kernel, as well as the associated point process. We set
\begin{equation}
\label{eq:truncatedcompact}
\tilde{K}^N (z_1,z_2) = \sum_{n = 0}^{N-1} \phi_n^N(z_1) \overline{\phi_n^N(z_2)}, \quad z_1, z_2 \in \mathcal{B}_R,
\end{equation}
and where $\phi_n^N$ corresponds to the function $\phi_n$ restricted to the compact $\mathcal{B}_{\sqrt{N}}$ (after renormalization). We emphasize that this is in fact $\mu^N |_{\mathcal{B}_{\sqrt{N}}}$ conditioned on there being $N$ points in the compact $\mathcal{B}_{\sqrt{N}}$, this result being due to Theorem~\ref{thm:thm7}. Moreover, the determinantal point process associated with this kernel benefits from the efficient simulations techniques developed in the previous subsection. Here, the fact that we can explicit the projection kernel associated with the conditioning is what ensures the efficiency of the simulation.
\\

\noindent
Let us start by proving that $\tilde{\mu}^N$, the associated determinantal process with kernel $\tilde{K}^N$, converges to $\mu$ weakly as $N$ tends to infinity. This is a consequence of Proposition~\ref{prop:convergencefaible}, as is proved in the following:
\begin{thm}
\label{thm:convergence}
We have that $\tilde{K}^N$ converges uniformly on compact subsets to $K$ as $N$ tends to infinity. As a consequence, the associated determinantal measures converge weakly to the determinantal point process of kernel $K$.
\end{thm}
\begin{proof}
Take a compact subset $A$ of $\C$, and write $|A| := \sup \{ | z |, \ z \in A \}$. Then, for $z_1, z_2 \in A$,
\begin{multline*}
| K(z_1,z_2) - \tilde{K}^N(z_1,z_2) | \le \sum_{n=0}^{N-1} \frac{1}{\pi} |A|^{2 n} |\frac{1}{\gamma(n+1,N)} - \frac{1}{n!}  | 1_{\{ (z_1,z_2) \in (\mathcal{B}_N)^2\} } \\
	+ \sum_{n = N}^\infty \frac{1}{\pi n!} |A|^{2 n} + \sum_{n=0}^{N-1} \frac{1}{\pi n!} | A |^{2 n} 1_{\{(z_1,z_2) \notin (\mathcal{B}_N)^2\}}.
\end{multline*}
The second term tends to zero as the remainder of a convergent series, and the third term also tends to zero by dominated convergence. Concerning the first term, we need sightly more precise arguments. Let us start by rewriting it as 
\begin{equation}
\label{eq:etapedemoconvergence}
\sum_{n=0}^{\infty} \frac{1}{\pi } |A|^{2 n} \frac{1}{\gamma(n+1,N)} 1_{\{(z_1,z_2)  \in (\mathcal{B}_N)^2\}} 1_{\{n \le N-1\}}  - \sum_{n=0}^{N-1}\frac{1}{\pi n!} |A|^{2 n} 1_{\{(z_1,z_2)  \in (\mathcal{B}_N)^2\}} ,
\end{equation}
and noticing that $\gamma(n+1,N) \rightarrow n!$ as $N$ tends to infinity. Therefore, in order to conclude, we wish to exhibit a summable bound. To this end, we write
\begin{align*}
\frac{1}{\gamma(n+1,N)}  1_{\{n \le N-1\}}  &\le \frac{1}{\gamma(n+1,n+1)}  \\		&=\frac{1}{n! \, \PP(\sum_{k=1}^{n+1} X_k \le n+1)} \\
	&\sim_{n \rightarrow \infty} \frac{2}{n!}  \\
\end{align*}
where $X_1,\dots,X_n$ are independent exponential random variables of parameter $1$. In the previous calculations, we have used the fact that $\frac{\gamma(a,R)}{\Gamma(a)}$ is the cumulative distribution function of a $\Gamma(a,R)$ random variable, $a > 0$, and $R \ge 0$. The last line results from the application of the central limit theorem to $X_1,\dots, X_n$. Hence, 
\begin{equation*}
\sum_{n=0}^{\infty} \frac{1}{\pi } |A|^{2 n} \frac{1}{\gamma(n+1,N)} 1_{\{(z_1,z_2)  \in (\mathcal{B}_N)^2\}} 1_{\{n \le N-1\}}  \le \sum_{n=0}^{\infty} \frac{1}{\pi \gamma(n+1,n+1) } |A|^{2 n}  < \infty,
\end{equation*}
which means that by Lebesgue's dominated convergence theorem, \eqref{eq:etapedemoconvergence} tends to zero as $N$ tends to infinity. Therefore, $| K(z_1,z_2) - \tilde{K}_N(z_1,z_2) | \xrightarrow[N \rightarrow \infty]{} 0$ for $z_1, z_2 \in A$. Hence, Proposition~\ref{prop:convergencefaible} allows us to conclude that $\tilde{\mu}^N \xrightarrow[N \rightarrow \infty]{\text{weakly}} \mu$.
\end{proof}

We now return to the problem of simulating the determinantal point process with kernel given by \eqref{eq:truncatedcompact}. As it is a projection process, it is efficiently simulated according to the basic algorithm described in Section~\ref{sec:algo}. On the other hand, the time-consuming step of generating the Bernoulli random variables is not necessary anymore, as we are working conditionally on there being $N$ points. Lastly, the method described in this section yields a determinantal point process on $\mathcal{B}_{\sqrt{N}}$. As before, in order to simulate on $\mathcal{B}_a$, we need to apply a homothetic transformation to the $N$ points, which translates to a homothety on the eigenvectors. To sum up, the simulation algorithm of the truncated Ginibre process on a centered ball of radius $a \ge 0$ is as follows:
\begin{algorithm}[h]
\caption{Simulation of the truncated Ginibre process on a compact}
 \begin{algorithmic}[H]
 \label{algo:simutruncated}
   \STATE {\bf define } $\phi_k(z) = \frac{N}{\pi a^2 \gamma(k+1,N)}e^{- \frac{N}{2 a^2} | z |^2} (\frac{N z}{a^2})^k $, for $z \in \mathcal{B}_N$ and $0\le k \le N-1$. \;
   \STATE {\bf define } $\bold v(z) := (\phi_0(z),\dots,\phi_{N-1}(z)  ) $, for $z \in \mathcal{B}_N$. \;
   \STATE {\bf sample } $X_N$ from the distribution with density $p_N(x) = \| \bold v(x) \|^2 /N$, $\ x \in \Lambda$ \;
   \STATE {\bf set } $\bold e_1 = \bold v(X_N) / \| \bold v(X_N) \|$
   \FOR {$i=N-1 \to 1$}
   \STATE {\bf sample } $X_i$ from the distribution with density \; 
   \begin{equation*}
   p_i(x) = \frac{1}{i} \Big[ \| \bold v(x) \|^2 - \sum_{j = 1}^{N-i} | \bold e_j^* \bold v(x) |^2 \Big]
   \end{equation*}
   \STATE {\bf set } $\bold w_i = \bold v(X_i) - \sum_{j=1}^{N-i} \left(\bold e_j^* \bold v(X_i)\right) \bold e_j,\quad \bold e_{N-i+1} = \bold w_i / \| \bold w_i \|$
   \ENDFOR
   \RETURN $(X_1,\dots,X_N)$
 \end{algorithmic}
\end{algorithm}

\noindent
The resulting process is a determinantal point process of kernel \eqref{eq:truncatedcompact}. Its support is on the compact $\mathcal{B}_a$ and has $N$ points almost surely. We now give a brief example of the results of the algorithm applied for $a = 2$ and $N = 9$ at steps $i=8$, $i=5$, and $i=2$ respectively. We have plotted the densities used for the simulation of the next point. We note here that the density is now supported on $\mathcal{B}_a$, whereas before the density was decreasing to zero outside of $\mathcal{B}_a$.

\begin{figure}[!htb]
\minipage{0.33\textwidth}
  \includegraphics[width=\linewidth]{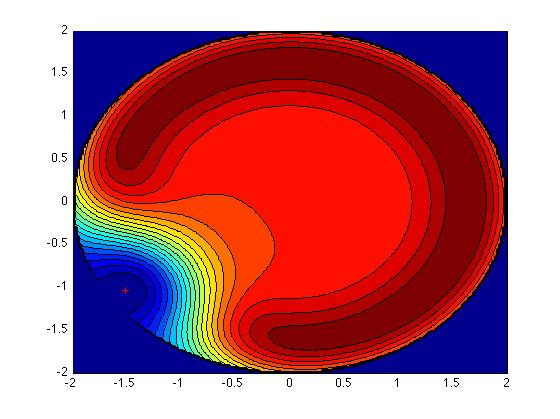}
\endminipage\hfill
\minipage{0.33\textwidth}
  \includegraphics[width=\linewidth]{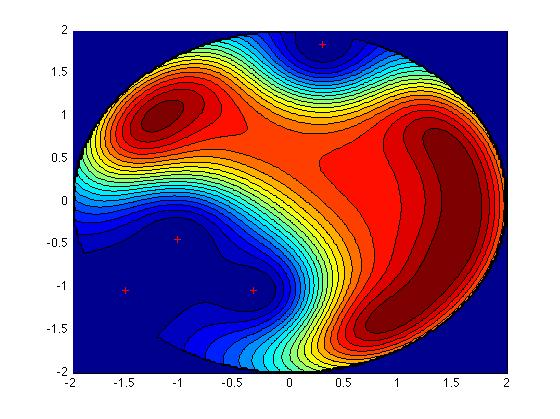}
\endminipage\hfill
\minipage{0.33\textwidth}%
  \includegraphics[width=\linewidth]{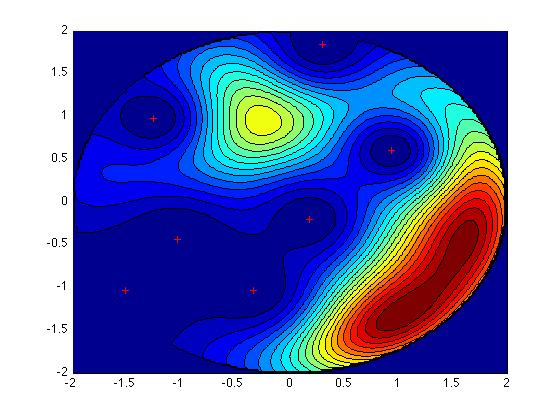}\label{fig:result2}
\endminipage
\end{figure}

\noindent
This determinantal point process presents the advantage of being easy to use in simulations, as well as having $N$ points almost surely. Moreover, Theorem~\ref{thm:convergence} proves its convergence to the Ginibre point process as $N$ tends to infinity.


\end{document}